\documentclass[12pt,reqno]{amsart}
\usepackage{amsfonts, amsthm, amsmath}
\allowdisplaybreaks[4]
\usepackage{rotating}
\usepackage{tikz}
\usepackage{graphics}
\usepackage{amssymb}
\usepackage{amscd}
\usepackage{t1enc}
\usepackage[mathscr]{eucal}
\usepackage{indentfirst}
\usepackage{graphicx}
\usepackage{graphics}
\usepackage{pict2e}
\usepackage{mathrsfs}
\usepackage{enumerate}
\usepackage[title,toc,titletoc]{appendix}
\usepackage{enumitem}
\usepackage[noadjust]{cite}
\usepackage{diagbox,color}
\usepackage{hyperref}
\hypersetup{backref, pagebackref, colorlinks=true}
\usepackage{cite}
\usepackage{color}
\usepackage{epic}
\usepackage{hyperref} 
\numberwithin{equation}{section}
\theoremstyle{plain}

\newtheorem{theorem}{Theorem}[section]
\newtheorem{lemma}[theorem]{Lemma}
\newtheorem{corollary}[theorem]{Corollary}

\theoremstyle{definition}

\newtheorem{remark}[theorem]{Remark}
\newtheorem{?}[theorem]{Problem}

\def\boxit#1{\leavevmode\hbox{\vrule\vtop{\vbox{\kern.33333pt\hrule
    \kern1pt\hbox{\kern1pt\vbox{#1}\kern1pt}}\kern1pt\hrule}\vrule}}

\usepackage{collectbox}



\newcommand{\f}[1]{\ifthenelse{\equal{#1}{1}}{(q;q)_\infty}{(q^{#1};q^{#1})_{\infty}}}

\setbox0=\hbox{$+$}
\newdimen\plusheight
\plusheight=\ht0
\def\+{\;\lower\plusheight\hbox{$+$}\;}

\setbox0=\hbox{$-$}
\newdimen\minusheight
\minusheight=\ht0
\def\-{\;\lower\minusheight\hbox{$-$}\;}

\setbox0=\hbox{$\cdots$}
\newdimen\cdotsheight
\cdotsheight=\plusheight
\def\cds{\lower\cdotsheight\hbox{$\cdots$}}

\begin{document}

\title[Ramanujan's parameter and its companion]{5-Dissections and sign patterns of Ramanujan's parameter and its companion}

\author[S. Chern]{Shane Chern}
\address[Shane Chern]{Department of Mathematics, Penn State University, University Park, PA 16802, USA}
\email{shanechern@psu.edu}

\author[D. Tang]{Dazhao Tang}
\address[Dazhao Tang]{Center for Applied Mathematics, Tianjin University, Tianjin 300072, P.R. China}
\email{dazhaotang@sina.com}

\date{\today}

\begin{abstract}	
In 1998, Michael Hirschhorn discovered 5-dissections of the Rogers--Ramanujan continued fraction $R(q)$ and its reciprocal. In this paper, we obtain the 5-dissections for functions $R(q)R(q^2)^2$ and $R(q)^2/R(q^2)$, which are essentially Ramanujan's parameter and its companion. 5-Dissections of the reciprocals of these two functions are derived as well. These 5-dissections imply that the coefficients in their series expansions have periodic sign patterns with few exceptions.
\end{abstract}



\maketitle

\section{Introduction}

Throughout, we always assume that $q$ is a complex number such that $|q|<1$. As a $q$-analog of the golden ratio, the Rogers--Ramanujan continued fraction
\begin{equation*}
R(q)=\frac{1}{1}\+\frac{q}{1}\+\frac{q^2}{1}\+\frac{q^3}{1}\+\cds
\end{equation*}
was discovered by Rogers \cite{Rog1894} and later popularized by Ramanujan \cite{Ram1957}. In some contexts like \cite{AB2005}, $R(q)$ has an extra factor of $q^{1/5}$. But in this paper, we will drop off this factor so that in its series expansion, all powers are nonnegative integers. The Rogers--Ramanujan continued fraction is closely related to the Rogers--Ramanujan identities, as indicated by their names. Before proceeding with our introduction, let us record some fairly standard notation:
\begin{align*}
(A;q)_\infty&:=\prod_{k= 0}^\infty (1-Aq^k),\\
(A_1,A_2,\ldots,A_n;q)_\infty&:=(A_1;q)_\infty (A_2;q)_\infty \cdots (A_n;q)_\infty,\\
\left(\begin{matrix}
A_1,A_2,\ldots,A_n\\
B_1,B_2,\ldots,B_m
\end{matrix};q\right)_\infty&:=\frac{(A_1;q)_\infty (A_2;q)_\infty \cdots (A_n;q)_\infty}{(B_1;q)_\infty (B_2;q)_\infty \cdots (B_m;q)_\infty}.
\end{align*}

The celebrated Rogers--Ramanujan identities (see \cite[Eqs.~(17.4.2) and (17.4.3)]{Hir2017}) state that
\begin{align}
G(q) &:=\sum_{n=0}^\infty\dfrac{q^{n^2}}{(1-q)(1-q^2)\cdots(1-q^n)}=\dfrac{1}{(q,q^4;q^5)_\infty},\label{RR-1}\\
H(q) &:=\sum_{n=0}^\infty\dfrac{q^{n^2+n}}{(1-q)(1-q^2)\cdots(1-q^n)}=\dfrac{1}{(q^2,q^3;q^5)_\infty}.\label{RR-2}
\end{align}
Here $G(q)$ and $H(q)$ are known as the Rogers--Ramanujan functions. The interrelation between $G(q)$ and $H(q)$ along with the continued fraction $R(q)$ was initially found by Rogers \cite{Rog1894}:
\begin{equation}\label{RR-CF}
R(q)=\frac{H(q)}{G(q)}=\left(\begin{matrix}
q,q^4\\
q^2,q^3
\end{matrix};q^5\right)_\infty.
\end{equation}
It was also recorded as Entry 38 (iii) in Chapter 16 of Ramanujan's Notebook \cite[p.~204]{Ram1957}.

Another two important functions in the theory of Rogers--Ramanujan continued fraction are the so-called Ramanujan's parameter
\begin{equation}\label{eq:Ram-para}
k:=qR(q)R(q^2)^2
\end{equation}
and its companion $R(q)^2/R(q^2)$. These functions are extensively used in the modular equations related to the Rogers--Ramanujan continued fraction. See, for instance, Raghavan and Rangachari \cite{RR1989}, Kang \cite{Kan1999}, Gugg \cite{Gug2009} and Cooper \cite{Coo2009}; see also Chapter 40 of Hirschhorn's book \cite{Hir2017}. Also, there is a whole level $10$ theory that leads to series for $1/\pi$ that involves Ramanujan's parameter $k$. See Cooper \cite{Coo2012} or Chapter 10 of Cooper's book \cite{Coo2017} for a detailed discussion.

In 1968 or early 1969, George Szekeres observed a surprising phenomenon that in the series expansions of $R(q)$ and its reciprocal $R(q)^{-1}$, the coefficients are eventually periodic with period $5$. This observation was later confirmed in the asymptotic sense by Richmond and Szekeres \cite{RS1978}. In 1978, Andrews \cite{And1981} found some formulas of Ramanujan in the Lost Notebook, and used these to give a complete proof of Szekeres' observation. The study of this topic reached the climax in 1998, at which time Hirschhorn \cite{Hir1998} discovered the following explicit 5-dissections of $R(q)$ and $R(q)^{-1}$:
\begin{align*}
R(q)&= \frac{(q^{125};q^{125})_\infty}{(q^{5};q^{5})_\infty}\Bigg(\left(\begin{matrix}
q^{30},q^{95}\\
q^{15},q^{110}
\end{matrix};q^{125}\right)_\infty-q\left(\begin{matrix}
q^{20},q^{105}\\
q^{10},q^{115}
\end{matrix};q^{125}\right)_\infty\notag\\
&\quad+q^2\left(\begin{matrix}
q^{55},q^{70}\\
q^{35},q^{90}
\end{matrix};q^{125}\right)_\infty-q^{18}\left(\begin{matrix}
q^{5},q^{120}\\
q^{60},q^{65}
\end{matrix};q^{125}\right)_\infty-q^4\left(\begin{matrix}
q^{45},q^{80}\\
q^{40},q^{85}
\end{matrix};q^{125}\right)_\infty\Bigg)
\end{align*}
and
\begin{align*}
\frac{1}{R(q)}&= \frac{(q^{125};q^{125})_\infty}{(q^{5};q^{5})_\infty}\Bigg(\left(\begin{matrix}
q^{40},q^{85}\\
q^{20},q^{105}
\end{matrix};q^{125}\right)_\infty+q\left(\begin{matrix}
q^{60},q^{65}\\
q^{30},q^{95}
\end{matrix};q^{125}\right)_\infty\notag\\
&\quad-q^7\left(\begin{matrix}
q^{35},q^{90}\\
q^{45},q^{80}
\end{matrix};q^{125}\right)_\infty-q^{3}\left(\begin{matrix}
q^{10},q^{115}\\
q^{5},q^{120}
\end{matrix};q^{125}\right)_\infty-q^{14}\left(\begin{matrix}
q^{15},q^{110}\\
q^{55},q^{70}
\end{matrix};q^{125}\right)_\infty\Bigg).
\end{align*}
It is notable that Ramanujan indeed partially discovered the two 5-dissections in terms of Lambert series in his Lost Notebook \cite[p.~50]{Ram1988}. But he left the final punch, which can be done by the quintuple product identity, to Mike Hirschhorn.

Considering the significance of dissection formulas in the study of sign patterns in $q$-series expansions, a natural question is about the 5-dissections of Ramanujan's parameter, its companion and their reciprocals, since these functions also play an important role in the theory of Rogers--Ramanujan continued fraction. In this paper, we obtain the following 5-dissections.

\begin{theorem}\label{th:th-1}
	We have
	\begin{align}
	R(q)R(q^2)^2 &=\left(\begin{matrix}
	q^5,q^5,q^{25},q^{25},q^{25},q^{25},q^{45},q^{45}\\
	q^{10},q^{10},q^{10},q^{20},q^{30},q^{40},q^{40},q^{40}
	\end{matrix};q^{50}\right)_\infty\notag\\
	&\quad-q\left(\begin{matrix}
	q^5,q^5,q^{15},q^{25},q^{25},q^{35},q^{45},q^{45}\\
	q^{10},q^{10},q^{10},q^{20},q^{30},q^{40},q^{40},q^{40}
	\end{matrix};q^{50}\right)_\infty\notag\\
	&\quad-q^2\left(\begin{matrix}
	q^5,q^{15},q^{15},q^{25},q^{25},q^{35},q^{35},q^{45}\\
	q^{10},q^{20},q^{20},q^{20},q^{30},q^{30},q^{30},q^{40}
	\end{matrix};q^{50}\right)_\infty\notag\\
	&\quad+2q^3\left(\begin{matrix}
	q^5,q^5,q^{15},q^{25},q^{25},q^{35},q^{45},q^{45}\\
	q^{10},q^{10},q^{20},q^{20},q^{30},q^{30},q^{40},q^{40}
	\end{matrix};q^{50}\right)_\infty\notag\\
	&\quad+q^9\left(\begin{matrix}
	q^5,q^5,q^5,q^{15},q^{35},q^{45},q^{45},q^{45}\\
	q^{10},q^{20},q^{20},q^{20},q^{30},q^{30},q^{30},q^{40}
	\end{matrix};q^{50}\right)_\infty\label{5-dis-3}
	\end{align}
	and
	\begin{align}
	\dfrac{1}{R(q)R(q^2)^2} &=\left(\begin{matrix}
	q^{15},q^{15},q^{25},q^{25},q^{25},q^{25},q^{35},q^{35}\\
	q^{10},q^{20},q^{20},q^{20},q^{30},q^{30},q^{30},q^{40}
	\end{matrix};q^{50}\right)_\infty\notag\\
	&\quad+q\left(\begin{matrix}
	q^5,q^{15},q^{15},q^{15},q^{35},q^{35},q^{35},q^{45}\\
	q^{10},q^{10},q^{10},q^{20},q^{30},q^{40},q^{40},q^{40}
	\end{matrix};q^{50}\right)_\infty\notag\\
	&\quad+2q^2\left(\begin{matrix}
	q^5,q^{15},q^{15},q^{25},q^{25},q^{35},q^{35},q^{45}\\
	q^{10},q^{10},q^{20},q^{20},q^{30},q^{30},q^{40},q^{40}
	\end{matrix};q^{50}\right)_\infty\notag\\
	&\quad+q^3\left(\begin{matrix}
	q^5,q^5,q^{15},q^{25},q^{25},q^{35},q^{45},q^{45}\\
	q^{10},q^{10},q^{10},q^{20},q^{30},q^{40},q^{40},q^{40}
	\end{matrix};q^{50}\right)_\infty\notag\\
	&\quad+q^4\left(\begin{matrix}
	q^5,q^{15},q^{15},q^{25},q^{25},q^{35},q^{35},q^{45}\\
	q^{10},q^{20},q^{20},q^{20},q^{30},q^{30},q^{30},q^{40}
	\end{matrix};q^{50}\right)_\infty.\label{5-dis-4}
	\end{align}
\end{theorem}

\begin{theorem}\label{th:th-2}
	We have
	\begin{align}
	\dfrac{R(q)^2}{R(q^2)} &=\left(\begin{matrix}
	-q^5,-q^{10},-q^{10},-q^{10},-q^{15},-q^{15},-q^{15},-q^{20}\\
	q^5,q^{10},q^{10},q^{10},q^{15},q^{15},q^{15},q^{20}
	\end{matrix};q^{25}\right)_\infty\notag\\
	&\quad-2q\left(\begin{matrix}
	-q^5,-q^{10},-q^{10},-q^{15},-q^{15},-q^{20},-q^{25},-q^{25}\\
	q^5,q^5,q^5,q^{10},q^{15},q^{20},q^{20},q^{20}
	\end{matrix};q^{25}\right)_\infty\notag\\
	&\quad+4q^2\left(\begin{matrix}
	-q^5,-q^{10},-q^{10},-q^{15},-q^{15},-q^{20},-q^{25},-q^{25}\\
	q^5,q^5,q^{10},q^{10},q^{15},q^{15},q^{20},q^{20}
	\end{matrix};q^{25}\right)_\infty\notag\\
	&\quad-4q^3\left(\begin{matrix}
	-q^{10},-q^{10},-q^{15},-q^{15},-q^{25},-q^{25},-q^{25},-q^{25}\\
	q^5,q^5,q^5,q^{10},q^{15},q^{20},q^{20},q^{20}
	\end{matrix};q^{25}\right)_\infty\notag\\
	&\quad+2q^4\left(\begin{matrix}
	-q^5,-q^5,-q^{10},-q^{15},-q^{20},-q^{20},-q^{25},-q^{25}\\
	q^5,q^{10},q^{10},q^{10},q^{15},q^{15},q^{15},q^{20}
	\end{matrix};q^{25}\right)_\infty\label{5-dis-2}
	\end{align}
	and
	\begin{align}
	\dfrac{R(q^2)}{R(q)^2} &=\left(\begin{matrix}
	-q^5,-q^5,-q^5,-q^{10},-q^{15},-q^{20},-q^{20},-q^{20}\\
	q^5,q^5,q^5,q^{10},q^{15},q^{20},q^{20},q^{20}
	\end{matrix};q^{25}\right)_\infty\notag\\
	&\quad+2q\left(\begin{matrix}
	-q^5,-q^{10},-q^{10},-q^{15},-q^{15},-q^{20},-q^{25},-q^{25}\\
	q^5,q^5,q^5,q^{10},q^{15},q^{20},q^{20},q^{20}
	\end{matrix};q^{25}\right)_\infty\notag\\
	&\quad-4q^7\left(\begin{matrix}
	-q^5,-q^5,-q^{20},-q^{20},-q^{25},-q^{25},-q^{25},-q^{25}\\
	q^5,q^{10},q^{10},q^{10},q^{15},q^{15},q^{15},q^{20}
	\end{matrix};q^{25}\right)_\infty\notag\\
	&\quad-4q^3\left(\begin{matrix}
	-q^5,-q^5,-q^{10},-q^{15},-q^{20},-q^{20},-q^{25},-q^{25}\\
	q^5,q^5,q^{10},q^{10},q^{15},q^{15},q^{20},q^{20}
	\end{matrix};q^{25}\right)_\infty\notag\\
	&\quad-2q^4\left(\begin{matrix}
	-q^5,-q^5,-q^{10},-q^{15},-q^{20},-q^{20},-q^{25},-q^{25}\\
	q^5,q^{10},q^{10},q^{10},q^{15},q^{15},q^{15},q^{20}
	\end{matrix};q^{25}\right)_\infty.
	\label{5-dis-1}
	\end{align}
\end{theorem}

As direct consequences of the 5-dissection formulas \eqref{5-dis-3}--\eqref{5-dis-1}, the following sign patterns hold.

\begin{corollary}
	Define the sequences $\{\alpha(n)\}$, $\{\beta(n)\}$, $\{\gamma(n)\}$ and $\{\delta(n)\}$ by
	\begin{align*}
	\sum_{n=0}^\infty\alpha(n)q^n &=R(q)R(q^2)^2,\qquad \sum_{n=0}^\infty\beta(n)q^n =\dfrac{1}{R(q)R(q^2)^2},\\
	\sum_{n=0}^\infty\gamma(n)q^n &=\dfrac{R(q)^2}{R(q^2)},\qquad\qquad\; \sum_{n=0}^\infty\delta(n)q^n =\dfrac{R(q^2)}{R(q)^2}.
	\end{align*}
	For any $n\geq0$, we have
	\begin{align}
	&\alpha(n)\begin{cases}
	>0, &\;\textrm{if}~\;n\equiv0,3,4,6,7\pmod{10},\\ <0,
	&\;\textrm{if}~\;n\equiv1,2,5,8,9\pmod{10},
	\end{cases}\label{ineq-1}\\
	&\beta(n)\begin{cases}
	>0, &\;\textrm{if}~\;n\equiv0,1,2,3,4\pmod{10},\\ <0,
	&\;\textrm{if}~\;n\equiv5,6,7,8,9\pmod{10},
	\end{cases}\label{ineq-2}\\
	&\gamma(n)\begin{cases}
	>0, &\;\textrm{if}~\;n\equiv0,2,4\pmod{5},\\ <0,
	&\;\textrm{if}~\;n\equiv1,3\pmod{5},
	\end{cases}\label{ineq-3}\\
	&\delta(n)\begin{cases}
	>0, &\;\textrm{if}~\;n\equiv0,1\pmod{5},\\ <0,
	&\;\textrm{if}~\;n\equiv2,3,4\pmod{5},
	\end{cases}\label{ineq-4}
	\end{align}
	except for $\alpha(4)=\beta(5)=\delta(2)=0$.
\end{corollary}

\section{Lemmas}\label{sect:lemmas}

Let us record some necessary identities for our proofs. The first several are related to the Rogers--Ramanujan functions $G(q)$ and $H(q)$.

\begin{lemma}
	We have
	\begin{align}
	G(q)^2H(q^2)+G(q^2)H(q)^2 &=\dfrac{2G(q)G(q^2)^2}{(q^5;q^{10})_\infty^2},\label{useful-1}\\
	G(q)^2H(q^2)-G(q^2)H(q)^2 &=\dfrac{2qH(q)H(q^2)^2}{(q^5;q^{10})_\infty^2},\label{useful-2}\\
	\dfrac{G(q)^2H(q^2)-G(q^2)H(q)^2}{G(q)^2H(q^2)+G(q^2)H(q)^2}
	&=\dfrac{qH(q)H(q^2)^2}{G(q)G(q^2)^2}.\label{useful-3}
	\end{align}
\end{lemma}
\begin{proof}
	The identities \eqref{useful-1}--\eqref{useful-3} are (17.4.10)--(17.4.12) in \cite{Hir2017}, respectively.
\end{proof}

The next lemma gives three modular equations related to Ramanujan's parameter $k$ as in \eqref{eq:Ram-para}.

\begin{lemma}
	We have
	\begin{align}
	\frac{R(q)^2}{R(q^2)} &=\frac{1-k}{1+k},\label{RR-iden-1}\\
	\dfrac{\varphi(-q)^2}{\varphi(-q^5)^2} &=\frac{1-4k-k^2}{1-k^2},\label{RR-iden-2}\\
	\dfrac{\psi(q)^2}{q\psi(q^5)^2} &=\dfrac{1+k-k^2}{k}.\label{RR-iden-3}
	\end{align}
	Here $\varphi(q)$ and $\psi(q)$ are two of Ramanujan's classical theta functions defined by
	\begin{align*}
	\varphi(q) &:=\sum_{n=-\infty}^\infty q^{n^2}=\dfrac{(q^2;q^2)_\infty^5}
	{(q;q)_\infty^2(q^4;q^4)_\infty^2},\\
	\psi(q) &:=\sum_{n=0}^\infty q^{n(n+1)/2}=\dfrac{(q^2;q^2)_\infty^2}{(q;q)_\infty}.
	\end{align*}
\end{lemma}

\begin{proof}
	The identity \eqref{RR-iden-1} appears in \cite[Eq.~(1.9.29)]{AB2005}. For \eqref{RR-iden-2} and \eqref{RR-iden-3}, see Entry 1.8.2 (i) and (ii) of \cite{AB2005}, respectively. See also Chapters 40 and 41 in \cite{Hir2017}.
\end{proof}

Further, we need the following identity involving $\varphi(-q)$.
\begin{lemma}
	We have
	\begin{align}
	\varphi(-q)^2-\varphi(-q^5)^2 &=-\dfrac{4q(q;q)_\infty(q^{10};q^{10})_\infty^3}
	{(q^2;q^2)_\infty(q^5;q^5)_\infty}.\label{useful-iden-3}
	\end{align}
\end{lemma}
\begin{proof}
	The identity \eqref{useful-iden-3} follows from \cite[Eq.~(34.1.21)]{Hir2017}.
\end{proof}

Finally, we require the following 5-dissections.
\begin{lemma}\label{le:DX-TX}
	We have
	\begin{align}
	\dfrac{G(q^2)^2}{G(q)H(q)^2} &=G(q^5)G(q^{10})^2-qH(q^5)G(q^{10})^2\notag\\
	&\quad-2q^3H(q^5)G(q^{10})H(q^{10})+q^4G(q^5)H(q^{10})^2,\label{DX-iden-1}\\
	\dfrac{H(q^2)^2}{G(q)^2H(q)} &=H(q^5)G(q^{10})^2-2qG(q^5)G(q^{10})H(q^{10})\notag\\
	&\quad+q^3G(q^5)H(q^{10})^2+q^4H(q^5)H(q^{10})^2,\label{DX-iden-2}\\
	\dfrac{G(q)^2}{G(q^2)^2H(q^2)} &=G(q^5)^2G(q^{10})+2qG(q^5)H(q^5)G(q^{10})\notag\\
	&\quad+q^2G(q^5)^2H(q^{10})+q^4H(q^5)^2H(q^{10}),\label{TX-iden-1}\\
	\dfrac{H(q)^2}{G(q^2)H(q^2)^2} &=G(q^5)^2G(q^{10})+q^2H(q^5)^2G(q^{10})\notag\\
	&\quad+2q^3G(q^5)H(q^5)H(q^{10})-q^4H(q^5)^2H(q^{10}).\label{TX-iden-2}
	\end{align}
\end{lemma}

\begin{proof}
	The identities \eqref{DX-iden-1} and \eqref{DX-iden-2} follow from Theorems 1.1 and 1.2 in \cite{DX2019} together with \eqref{RR-1} and \eqref{RR-2}, respectively.
	For \eqref{TX-iden-1} and \eqref{TX-iden-2}, we make use of Theorems 1.1 and 1.2
	in \cite{TX2019}.
\end{proof}

\section{Auxiliary identities}\label{sect:Aux-iden}

We shall establish two auxiliary identities that are also of independent interest.

\begin{lemma}\label{le:key-iden-2}
	We have
	\begin{align}\label{key-iden-2}
	&G(q)^4H(q^2)^2-H(q)^4G(q^2)^2+4G(q)^2H(q)^2G(q^2)H(q^2)\notag\\
	&\qquad\qquad\qquad\qquad\qquad\qquad\qquad\qquad
	=\dfrac{4(q^{10};q^{10})_\infty^4}{(q^5;q^5)_\infty^4}
	\cdot\dfrac{G(q)^3H(q)^3}{G(q^2)^2H(q^2)^2}.
	\end{align}
\end{lemma}

\begin{proof}
	We first prove
	\begin{align}\label{eq:aux-step-2}
	\dfrac{R(q^2)}{R(q)^2}-\dfrac{R(q)^2}{R(q^2)}+4 &=
	\dfrac{4(q^2;q^2)_\infty^3(q^{10};q^{10})_\infty}{(q;q)_\infty(q^5;q^5)_\infty^3}.
	\end{align}
	To see this, we obtain from \eqref{RR-iden-1} that
	\begin{align}
	&\dfrac{R(q^2)}{R(q)^2}-\dfrac{R(q)^2}{R(q^2)}+4=\frac{1+k}{1-k}-\frac{1-k}{1+k}+4
	=\dfrac{4k}{1-k^2}\cdot\dfrac{1+k-k^2}{k}.\label{iden-2-1}
	\end{align}
	Further, from \eqref{RR-iden-2} we find that
	\begin{align*}
	\dfrac{\varphi(-q)^2}{\varphi(-q^5)^2} &=\frac{1-4k-k^2}{1-k^2}=1-\dfrac{4k}{1-k^2}.
	\end{align*}
	Hence,
	\begin{align}
	\dfrac{4k}{1-k^2}=-\dfrac{\varphi(-q)^2-\varphi(-q^5)^2}
	{\varphi(-q^5)^2}=\dfrac{4q(q;q)_\infty(q^{10};q^{10})_\infty^5}
	{(q^2;q^2)_\infty(q^5;q^5)_\infty^5},\label{useful-iden-4}
	\end{align}
	where \eqref{useful-iden-3} is utilized. We conclude \eqref{eq:aux-step-2} by substituting \eqref{RR-iden-3} and \eqref{useful-iden-4} into \eqref{iden-2-1}. Finally, we notice that
	\begin{align*}
	&G(q)^4H(q^2)^2-H(q)^4G(q^2)^2+4G(q)^2H(q)^2G(q^2)H(q^2)\\
	&\quad=G(q)^2H(q)^2G(q^{2})H(q^{2})\left(\dfrac{G(q)^2H(q^2)}{G(q^2)H(q)^2}
	-\dfrac{G(q^2)H(q)^2}{G(q)^2H(q^2)}+4\right)\\
	&\quad=G(q)^2H(q)^2G(q^{2})H(q^{2})\left(\dfrac{R(q^2)}{R(q)^2}-\dfrac{R(q)^2}{R(q^2)}
	+4\right).
	\end{align*}
	Substituting \eqref{eq:aux-step-2} into the above yields the desired identity.
\end{proof}

\begin{lemma}\label{le:key-iden-1}
	We have
	\begin{align}\label{key-iden-1}
	H(q)^4G(q^{2})^2-G(q)^4H(q^{2})^2+G(q)^2H(q)^2G(q^{2})H(q^{2})=
	\frac{G(q^{2})^3H(q^{2})^3}{G(q)^2H(q)^2}.
	\end{align}
\end{lemma}

\begin{proof}
	We first prove
	\begin{equation}\label{eq:aux-step-1}
	\dfrac{R(q)^2}{R(q^2)}
	-\dfrac{R(q^2)}{R(q)^2}+1 = \dfrac{\varphi(-q)^2}{\varphi(-q^5)^2}.
	\end{equation}
	It follows from \eqref{RR-iden-1} that
	\begin{align*}
	\dfrac{R(q)^2}{R(q^2)}
	-\dfrac{R(q^2)}{R(q)^2}+1=\frac{1-k}{1+k}-\frac{1+k}{1-k}+1=\frac{1-4k-k^2}{1-k^2}.
	\end{align*}
	Therefore, \eqref{eq:aux-step-1} is a direct consequence of \eqref{RR-iden-2}. Finally,
	\begin{align*}
	&H(q)^4 G(q^{2})^2 -G(q)^4 H(q^{2})^2 + G(q)^2H(q)^2G(q^{2})H(q^{2})\\
	&\quad=G(q)^2H(q)^2G(q^{2})H(q^{2})\left(\dfrac{G(q^2)H(q)^2}{G(q)^2H(q^2)}
	-\dfrac{G(q)^2H(q^2)}{G(q^2)H(q)^2}+1\right)\\
	&\quad=G(q)^2H(q)^2G(q^{2})H(q^{2})\left(\dfrac{R(q)^2}{R(q^2)}
	-\dfrac{R(q^2)}{R(q)^2}+1\right).
	\end{align*}
	We arrive at the desired identity by substituting \eqref{eq:aux-step-1} into the above.
\end{proof}

\section{Proofs}

\subsection{Idea behind the proofs}

Before presenting our proofs, it appears necessary to explain the idea behind them.

\begin{itemize}[leftmargin=*]\itemsep3pt
	\item \textbf{Guessing the 5-dissection.} Given a formal power series $\sum_{n=0}^\infty a_nq^n$, we may start by computing, for example, the first 1000 terms of each $m$-dissection slice $\sum_{n=0}^\infty a_{mn+\ell}q^n$ (with $0\le \ell\le m-1$). If the slice has a nice product form, then such a $q$-product could be induced through the \texttt{prodmake} command of Garvan's \textit{Maple} package \texttt{qseries} \cite{Gar1999}. In other words, we can take advantage of the package \texttt{qseries} to conjecture the 5-dissections of $R(q)R(q^2)^2$, $1/R(q)R(q^2)^2$, $R(q)^2/R(q^2)$ and $R(q^2)/R(q)^2$.
	
	In a recent work, Frye and Garvan \cite{FG2019} further implemented another two \textit{Maple} packages \texttt{thetaids} and \texttt{ramarobinsids} to provide automatic proofs of theta-function identities with the help of modular forms. It should be admitted that the above conjectured 5-dissections could be shown automatically through the package \texttt{thetaids}. The interested reader may consult \cite{XZ2020} for a recent application of this computer-assisted approach.
	
	On the other hand, elementary proof of dissection formulas often has its own research interest. In this section, we will present such proofs of \eqref{5-dis-3}--\eqref{5-dis-1} in which only identities involving the Rogers--Ramanujan functions $G(q)$ and $H(q)$ will be utilized. In particular, the trick of \textit{substitution of one} which will be discussed below makes our proof less routine.
	
	\item \textbf{Constructing an auxiliary product.} To start the proofs, we first need to reformulate $R(q)R(q^2)^2$, $1/R(q)R(q^2)^2$, $R(q)^2/R(q^2)$ and $R(q^2)/R(q)^2$ as
	\begin{align*}
	R(q)R(q^2)^2 &=\dfrac{(q^2,q^3;q^5)_\infty(q,q^9;q^{10})_\infty^2}	{(q,q^4;q^5)_\infty(q^3,q^7;q^{10})_\infty^2}
=\dfrac{H(q^2)^2}{G(q)^2H(q)}\bigg/\dfrac{G(q^2)^2}{G(q)H(q)^2},\\
	\dfrac{1}{R(q)R(q^2)^2} &=\dfrac{(q,q^4;q^5)_\infty(q^3,q^7;q^{10})_\infty^2}	{(q^2,q^3;q^5)_\infty(q,q^9;q^{10})_\infty^2}
=\dfrac{G(q^2)^2}{G(q)H(q)^2}\bigg/\dfrac{H(q^2)^2}{G(q)^2H(q)},\\
	\dfrac{R(q)^2}{R(q^2)} &=\dfrac{(q,q^4;q^5)_\infty^2(q^4,q^6;q^{10})_\infty}	{(q^2,q^3;q^5)_\infty^2(q^2,q^8;q^{10})_\infty}
=\dfrac{H(q)^2}{G(q^2)H(q^2)^2}\bigg/\dfrac{G(q)^2}{G(q^2)^2H(q^2)},\\
	\dfrac{R(q^2)}{R(q)^2} &=\dfrac{(q^2,q^3;q^5)_\infty^2(q^2,q^8;q^{10})_\infty}	{(q,q^4;q^5)_\infty^2(q^4,q^6;q^{10})_\infty}
=\dfrac{G(q)^2}{G(q^2)^2H(q^2)}\bigg/\dfrac{H(q)^2}{G(q^2)H(q^2)^2}.
	\end{align*}
	As long as we have conjectured the 5-dissections through Garvan's \textit{Maple} package \texttt{qseries}, then in light of the above, it suffices to prove
\begin{align*}
\textrm{Dissection} = \textrm{Numerator}/\textrm{Denominator},
\end{align*}
	or
\begin{align*}
\textrm{Denominator}\times \textrm{Dissection} = \textrm{Numerator},
\end{align*}
	where the ``$\textrm{Numerator}$'' and ``$\textrm{Denominator}$'' come from the right-hand sides of the above reformulations. Now we multiply by a simple extra term on both sides of the above,
\begin{align*}
\textrm{Denominator}\times \textrm{Dissection}\times [\textrm{extra term}]
=\textrm{Numerator}\times [\textrm{extra term}].
\end{align*}
Our auxiliary products then come from here. More precisely, these auxiliary products have the form
\begin{align*}
\Pi_*=\Pi_{*1}\times \Pi_{*2}
\end{align*}
	such that $\Pi_{*1}=\textrm{Denominator}$ by Lemma \ref{le:DX-TX} and $\Pi_{*2}= \textrm{Dissection}\times [\textrm{extra term}]$ by \eqref{useful-1} and \eqref{useful-2}.
Therefore, it suffices to verify
\begin{align*}
\Pi_*\stackrel{?}{=} \textrm{Numerator}\times [\textrm{extra term}].
\end{align*}
	
	\item \textbf{Substitution of one.} To complete the proofs, we want to highlight a trick which we call the \textit{substitution of one}. We first expand the auxiliary product $\Pi_*=\Pi_{*1}\times \Pi_{*2}$ and collect terms according to the power of $q$,
\begin{align*}
\Pi_*=\Sigma_{*1}+q^5\Sigma_{*2}
\end{align*}
	such that in both $\Sigma_{*1}$ and $\Sigma_{*2}$, the powers of $q$ are up to $q^4$. We then see from \eqref{useful-3} that
	\begin{align}\label{eq:hir-17.4.12}
	\frac{G(q^5)G(q^{10})^2}{q^5H(q^5)H(q^{10})^2}\cdot
	\frac{G(q^5)^2H(q^{10})-H(q^5)^2G(q^{10})}{G(q^5)^2H(q^{10})+H(q^5)^2G(q^{10})}
	=1.
	\end{align}
	Hence, we may rewrite $\Pi_*$ as
\begin{align*}
\Pi_* =\Sigma_{*1}+q^5\Sigma_{*2}\times\Bigg(\frac{G(q^5)G(q^{10})^2}{q^5H(q^5)H(q^{10})^2}\cdot
	\frac{G(q^5)^2H(q^{10})-H(q^5)^2G(q^{10})}{G(q^5)^2H(q^{10})+H(q^5)^2G(q^{10})}\Bigg).
\end{align*}
	Expanding and then factoring the above, we arrive at an expression where either Lemma \ref{le:key-iden-2} or Lemma \ref{le:key-iden-1} could be applied. Finally, we deduce $\Pi_*= \textrm{Numerator}\times [\textrm{extra term}]$ through Lemma \ref{le:DX-TX}.
\end{itemize}

\begin{remark}
	It is notable that the trick of \textit{substitution of one} has also been used in \cite{CT2020, Tan2020}. However, in that work, $1$ is substituted by a different expression of the Rogers--Ramanujan functions $G(q)$ and $H(q)$ instead of \eqref{eq:hir-17.4.12}.
\end{remark}

We will use \eqref{5-dis-3} to present a detailed proof while the rest could be shown analogously.

\subsection{Proof of (\ref{5-dis-3})}

We define the auxiliary product by $\Pi_1:=\Pi_{11}\times \Pi_{12}$ where
\begin{align*}
\Pi_{11}&=G(q^5)G(q^{10})^2-qH(q^5)G(q^{10})^2\\
&\quad-2q^3H(q^5)G(q^{10})H(q^{10})+q^4G(q^5)H(q^{10})^2
\end{align*}
and
\begin{align*}
\Pi_{12}&=4G(q^5)^2H(q^5)^3G(q^{10})^3H(q^{10})^3\notag\\
&\quad-2qG(q^5)H(q^5)^2G(q^{10})^2H(q^{10})^3\cdot
\left(G(q^5)^2H(q^{10})+H(q^5)^2G(q^{10})\right)\notag\\
&\quad-2q^2G(q^5)^2H(q^5)G(q^{10})H(q^{10})^4\cdot
\left(G(q^5)^2H(q^{10})+H(q^5)^2G(q^{10})\right)\notag\\
&\quad+4q^3G(q^5)H(q^5)^2G(q^{10})H(q^{10})^4\cdot
\left(G(q^5)^2H(q^{10})+H(q^5)^2G(q^{10})\right)\notag\\
&\quad+q^4H(q^5)H(q^{10})^4\cdot
\left(G(q^5)^2H(q^{10})+H(q^5)^2G(q^{10})\right)\notag\\
&\quad\quad\times\left(G(q^5)^2H(q^{10})-H(q^5)^2G(q^{10})\right).
\end{align*}
Then by \eqref{DX-iden-1}, we have
\begin{align}\label{Pi_11}
\Pi_{11}=\dfrac{G(q^2)^2}{G(q)H(q)^2}.
\end{align}
Also, in light of \eqref{useful-1} and \eqref{useful-2},
\begin{align}\label{Pi_12}
\Pi_{12} &=4G(q^5)^4H(q^5)^3\notag\\
&\quad\times\Bigg(\dfrac{G(q^{10})^3H(q^{10})^3}{G(q^5)^2}
-\dfrac{qG(q^{10})^4H(q^{10})^3}{G(q^5)^2H(q^5)(q^{25};q^{50})_\infty^2}
-\dfrac{q^2G(q^{10})^3H(q^{10})^4}{G(q^5)H(q^5)^2(q^{25};q^{50})_\infty^2}\notag\\
&\quad\quad\quad+\dfrac{2q^3G(q^{10})^3H(q^{10})^4}{G(q^5)^2H(q^5)(q^{25};q^{50})_\infty^2}
+\dfrac{q^9G(q^{10})^2H(q^{10})^6}{G(q^5)^3H(q^5)(q^{25};q^{50})_\infty^4}\Bigg).
\end{align}

On the other hand, we expand the product in $\Pi_1$ and rearrange terms so that $\Pi_1 = \Sigma_{11}+q^5\Sigma_{12}$ where
\begin{align*}
\Sigma_{11}&=4G(q^5)^3H(q^5)^3G(q^{10})^5H(q^{10})^3\\
&\quad-2q\big(3G(q^5)^2H(q^5)^4G(q^{10})^5H(q^{10})^3
+G(q^5)^4H(q^5)^2G(q^{10})^4H(q^{10})^4\big)\\
&\quad+2q^2\big(G(q^5)H(q^5)^5G(q^{10})^5H(q^{10})^3
-G(q^5)^5H(q^5)G(q^{10})^3H(q^{10})^5\big)\\
&\quad-2q^3\big(G(q^5)^2H(q^5)^4G(q^{10})^4H(q^{10})^4
-3G(q^5)^4H(q^5)^2G(q^{10})^3H(q^{10})^5\big)\\
&\quad-q^4\big(G(q^5)H(q^5)^5G(q^{10})^4H(q^{10})^4
-4G(q^5)^3H(q^5)^3G(q^{10})^3H(q^{10})^5\\
&\qquad-G(q^5)^5H(q^5)G(q^{10})^2H(q^{10})^6\big)\\
\intertext{and}
\Sigma_{12}&=\big(H(q^5)^6G(q^{10})^4H(q^{10})^4
+2G(q^5)^2H(q^5)^4G(q^{10})^3H(q^{10})^5\\
&\qquad+G(q^5)^4H(q^5)^2G(q^{10})^2H(q^{10})^6\big)\\
&\quad-2q\big(4G(q^5)H(q^5)^5G(q^{10})^3H(q^{10})^5+
5G(q^5)^3H(q^5)^3G(q^{10})^2H(q^{10})^6\\
&\qquad+G(q^5)^5 H(q^5)G(q^{10})H(q^{10})^7\big)\\
&\quad+2q^2\big(H(q^5)^6G(q^{10})^3H(q^{10})^5
+2G(q^5)^2H(q^5)^4G(q^{10})^2H(q^{10})^6\\
&\qquad+G(q^5)^4H(q^5)^2G(q^{10})H(q^{10})^7\big)\\
&\quad-q^3\big(G(q^5)H(q^5)^5G(q^{10})^2H(q^{10})^6
-G(q^5)^5H(q^5)H(q^{10})^8\big).
\end{align*}
By \eqref{eq:hir-17.4.12}, we rewrite $\Pi_1$ as
\begin{align*}
\Pi_1 &=\Sigma_{11}+q^5\Sigma_{12}\left(\frac{G(q^5)G(q^{10})^2}{q^5H(q^5)H(q^{10})^2}\cdot
\frac{G(q^5)^2H(q^{10})-H(q^5)^2G(q^{10})}{G(q^5)^2H(q^{10})+H(q^5)^2G(q^{10})}\right)\\
&=G(q^5)G(q^{10})^2H(q^{10})^2\\
&\quad\times\left(G(q^5)^4H(q^{10})^2-H(q^5)^4G(q^{10})^2
+4G(q^5)^2H(q^5)^2G(q^{10})H(q^{10})\right)\\
&\quad\times\big(G(q^{10})^2H(q^5)-2qG(q^5)G(q^{10})H(q^{10})\\
&\qquad+q^3G(q^5)H(q^{10})^2+q^4H(q^5)H(q^{10})^2\big).
\end{align*}
Applying \eqref{key-iden-2} to the second factor and \eqref{DX-iden-2} to the third factor, we find that
\begin{align}\label{Pi-1-iden-1}
\Pi_1 &=\dfrac{4(q^{50};q^{50})_\infty^4}{(q^{25};q^{25})_\infty^4}
\cdot G(q^5)^4H(q^5)^3\cdot\dfrac{H(q^2)^2}{G(q)^2H(q)}.
\end{align}

Therefore, \eqref{5-dis-3} follows (after simplification) from \eqref{Pi_11}, \eqref{Pi_12} and \eqref{Pi-1-iden-1} together with the fact that
$$R(q)R(q^2)^2 =\dfrac{H(q^2)^2}{G(q)^2H(q)}\bigg/\dfrac{G(q^2)^2}{G(q)H(q)^2}.$$

\subsection{The remaining dissections}

We list the required auxiliary products in the proofs of \eqref{5-dis-4}, \eqref{5-dis-2} and \eqref{5-dis-1}. The calculations are similar to those for \eqref{5-dis-3} and therefore the details are omitted.

\begin{itemize}[leftmargin=*]\itemsep3pt
	\item \textit{Dissection \eqref{5-dis-4}}. We require the auxiliary product $\Pi_2 = \Pi_{21}\times \Pi_{22}$ where
	\begin{align*}
	\Pi_{21}&=H(q^5)G(q^{10})^2-2qG(q^5)G(q^{10})H(q^{10})\\
	&\quad
	+q^3G(q^5)H(q^{10})^2+q^4H(q^5)H(q^{10})^2
	\end{align*}
	and
	\begin{align*}
	\Pi_{22}&=4G(q^5)^4H(q^5)G(q^{10})^3H(q^{10})^3\notag\\
	&\quad+qG(q^5)G(q^{10})^2H(q^{10})^2\cdot
	\left(G(q^5)^2H(q^{10})+H(q^5)^2G(q^{10})\right)^2\notag\\
	&\quad+4q^2G(q^5)^2H(q^5)G(q^{10})^2H(q^{10})^3\cdot
	\left(G(q^5)^2H(q^{10})+H(q^5)^2G(q^{10})\right)\notag\\
	&\quad+2q^3G(q^5)H(q^5)^2G(q^{10})^2H(q^{10})^3\cdot
	\left(G(q^5)^2H(q^{10})+H(q^5)^2G(q^{10})\right)\notag\\
	&\quad+2q^4G(q^5)^2H(q^5)G(q^{10})H(q^{10})^4\cdot
	\left(G(q^5)^2H(q^{10})+H(q^5)^2G(q^{10})\right).
	\end{align*}
	
	\item \textit{Dissection \eqref{5-dis-2}}. We require the auxiliary product $\Pi_3 = \Pi_{31}\times \Pi_{32}$ where
	\begin{align*}
	\Pi_{31}&=G(q^5)^2G(q^{10})+2qG(q^5)H(q^5)G(q^{10})\\
	&\quad+q^2G(q^5)^2H(q^{10})+q^4H(q^5)^2H(q^{10})
	\end{align*}
	and
	\begin{align*}
	\Pi_{32}&=G(q^5)^2H(q^5)^6G(q^{10})^3\notag\\
	&\qquad-qG(q^5)^3H(q^5)^3G(q^{10})H(q^{10})\cdot
	\left(G(q^5)^2H(q^{10})+H(q^5)^2G(q^{10})\right)\notag\\	
	&\qquad+2q^2G(q^5)^2H(q^5)^4G(q^{10})H(q^{10})
	\cdot\left(G(q^5)^2H(q^{10})+H(q^5)^2G(q^{10})\right)\notag\\	
	&\qquad-q^3G(q^5)H(q^5)^3H(q^{10})
	\cdot\left(G(q^5)^2H(q^{10})+H(q^5)^2G(q^{10})\right)^2\notag\\ &\qquad+q^4G(q^5)^2H(q^5)^4H(q^{10})^2\cdot\left(G(q^5)^2H(q^{10})
	+H(q^5)^2G(q^{10})\right).
	\end{align*}
	
	\item \textit{Dissection \eqref{5-dis-1}}. We require the auxiliary product $\Pi_4 = \Pi_{41}\times \Pi_{42}$ where
	\begin{align*}
	\Pi_{41}&=G(q^5)^2G(q^{10})+q^2H(q^5)^2G(q^{10})\\
	&\quad+2q^3G(q^5)H(q^5)H(q^{10})-q^4H(q^5)^2H(q^{10})
	\end{align*}
	and
	\begin{align*}
	\Pi_{42}&=G(q^5)^6H(q^5)^2G(q^{10})H(q^{10})^2\notag\\
	&\qquad+qG(q^5)^3H(q^5)^3G(q^{10})H(q^{10})\cdot\left(G(q^5)^2H(q^{10})
	+H(q^5)^2G(q^{10})\right)\notag\\
	&\qquad-q^2G(q^5)^2H(q^5)^2H(q^{10})\cdot\left(G(q^5)^2H(q^{10})
	+H(q^5)^2G(q^{10})\right)\notag\\
	&\qquad\quad\times
	\left(G(q^5)^2H(q^{10})-H(q^5)^2G(q^{10})\right)\notag\\
	&\qquad-2q^3G(q^5)^3H(q^5)^3H(q^{10})^2
	\cdot\left(G(q^5)^2H(q^{10})+H(q^5)^2G(q^{10})\right)\notag\\
	&\qquad-q^4G(q^5)^2H(q^5)^4H(q^{10})^2
	\cdot\left(G(q^5)^2H(q^{10})+H(q^5)^2G(q^{10})\right).
	\end{align*}
\end{itemize}

\subsection*{Acknowledgements}

The authors would like to thank Mike Hirschhorn for some helpful suggestions. The second author was supported by the Postdoctoral Science Foundation of China (No.~2019M661005).

\bibliographystyle{amsplain}

\begin{thebibliography}{99}
	
	\bibitem{And1981}
	G. E. Andrews: Ramanujan's ``lost'' notebook. III. The Rogers--Ramanujan continued fraction,
    \textit{Adv. in Math.} 41 (1981), 186--208. 
	
	\bibitem{AB2005}
	G. E. Andrews, B. C. Berndt: \textit{Ramanujan's Lost Notebook. Part I}, Springer, New
    York, 2005. 
    
    \bibitem{CT2020}
    S. Chern, D. Tang: Vanishing coefficients in quotients of theta functions of modulus five,
    \textit{Bull. Aust. Math. Soc.} (2020), in press. doi: 10.1017/S0004972720000271.

    \bibitem{Coo2009}
    S. Cooper: On Ramanujan's function $k(q)=r(q)r^2(q^2)$, \textit{Ramanujan J.} 20 (2009), 311--328.
    

    \bibitem{Coo2012}
    S. Cooper: Level 10 analogues of Ramanujan's series for $1/\pi$, \textit{J. Ramanujan Math. Soc.}
    27 (2012), 59--76. 
    
    
    \bibitem{Coo2017}
    S. Cooper: \textit{Ramanujan's Theta Functions}, Springer, Cham, 2017. 

	\bibitem{DX2019}
    D. Q. J. Dou, J. Xiao: The 5-dissections of two infinite product expansions, \textit{Ramanujan J.}
    (2020), in press. doi: 10.1007/s11139-019-00200-w.

    \bibitem{FG2019}
    J. Frye, F. Garvan: Automatic proof of theta-function identities, in: Elliptic integrals, elliptic functions and modular forms in quantum field theory, in: Texts Monogr. Symbol. Comput., Springer, Cham, 2019, pp. 195--258. 

    \bibitem{Gar1999}
    F. Garvan: A $q$-product tutorial for a $q$-series MAPLE package, The Andrews Festschrift
    (Maratea, 1998), \textit{S\'em. Lothar. Combin.} 42 (1999), Art. B42d, 27 pp. 

    \bibitem{Gug2009}
    C. Gugg: Two modular equations for squares of the Rogers--Ramanujan functions with applications, \textit{Ramanujan J.} 18 (2009), 183--207. 

	\bibitem{Hir1998}
    M. D. Hirschhorn: On the expansion of Ramanujan's continued fraction, \textit{Ramanujan J.} 2 (1998),
    521--527.

	\bibitem{Hir2017}
    M. D. Hirschhorn: \textit{The Power of $q$}. A Personal Journey, Developments in Mathematics, vol. 49. Springer, Cham, 2017.

    \bibitem{Kan1999}
    S.-Y. Kang: Some theorems on the Rogers-Ramanujan continued fraction and associated theta function identities in Ramanujan's lost notebook, \textit{Ramanujan J.} 3 (1999), 91--111.

    \bibitem{RR1989}
    S. Raghavan, S. S. Rangachari: Ramanujan's elliptic integrals and modular identities, in:
    Number Theory and Related Topics, Oxford University Press, Bombay, 1989, pp. 119--149.
    

	\bibitem{Ram1957}
    S. Ramanujan: \textit{Notebooks, Vol. II}, Tata Institute of Fundamental Research, Bombay 1957. Reprinted by Springer-Verlag, Berlin, 1984.


	\bibitem{Ram1988}
    S. Ramanujan: \textit{The Lost Notebook and Other Unpublished Papers}, Springer-Verlag, Berlin; Narosa Publishing House, New Delhi, 1988. 

	\bibitem{RS1978}
    B. Richmond, G. Szekeres: The Taylor coefficients of certain infinite products, \textit{Acta Sci. Math. (Szeged)} 40 (1978), no. 3-4, 347--369.

	\bibitem{Rog1894}
    L. J. Rogers: Second memoir on the expansion of certain infinite products, \textit{Proc. Lond. Math. Soc.} 25 (1894), 318--343.

    \bibitem{Tan2020}
    D. Tang: On 5- and 10-dissections for some infinite products, to appear in \textit{Ramanujan J}.

	\bibitem{TX2019}
    D. Tang, E. X. W. Xia: Several $q$-series related to Ramanujan's theta functions, \textit{Ramanujan J.} (2019), in press. doi: 10.1007/s11139-019-00187-4.

    \bibitem{XZ2020}
    E. X. W. Xia, A. X. H. Zhao: Generalizations of Hirschhorn's results on two remarkable $q$-series expansions, \textit{Exp. Math.} (2020), in press. doi: 10.1080/10586458.2020.1712565.

	

\end{thebibliography}

\end{document}